\documentclass[12pt,a4paper]{amsart}
\usepackage{amsmath,amsfonts,amssymb,amsthm,amscd}
\usepackage[english]{babel}

\input{comment.sty}
\includecomment{MM}
\includecomment{CL}

\renewcommand{\a}{\alpha}
\renewcommand{\b}{\beta}
\newcommand{\g}{\gamma}
\newcommand{\G}{\Gamma}
\renewcommand{\d}{\delta}

\renewcommand{\i}{\iota}

\renewcommand{\O}{\Omega}
\renewcommand{\r}{\rho}
\newcommand{\s}{\sigma}

\renewcommand{\t}{\tau}

\newcommand{\e}{\varepsilon}
\newcommand{\f}{\varphi}

\newcommand{\Ac}{{\mathcal A}}

\newcommand{\Dc}{{\mathcal D}}
\newcommand{\Pc}{{\mathcal P}}

\newcommand{\N}{{\mathbb N}}

\newtheorem{theorem}{Theorem}[section]
\newtheorem{proposition}[theorem]{Proposition}
\newtheorem{lemma}[theorem]{Lemma}
\newtheorem{corollary}[theorem]{Corollary}

\newtheorem{fact}[theorem]{Fact}

\theoremstyle{definition}
\newtheorem{definition}[theorem]{Definition}

\newtheorem{notation}[theorem]{Notation}
\newtheorem{remark}[theorem]{Remark}

\theoremstyle{remark}

\newtheorem{question}[theorem]{Question}

\begin{document}

\title[Folner sets of alternate directed groups]
      {Folner sets of alternate directed groups}
\author[Brieussel]{J\'er\'emie Brieussel \\ \\ Universit\'e Montpellier 2 \\ France}
\email{jeremie.brieussel@univ-montp2.fr}

\date{April 12, 2013.}

\begin{abstract}
An explicit family of Folner sets is constructed for some directed groups acting on a rooted tree of sublogarithmic valency by alternate permutations. In the case of bounded valency, these groups were known to be amenable by probabilistic methods. The present construction provides a new and independent proof of amenability, using neither random walks, nor word length.
\end{abstract}

\maketitle

\section{Introduction}

By a criterion of Folner \cite{Fol}, amenable groups are those that admit finite subsets with arbitrary small boundaries relatively to their cardinality. A sequence of such subsets, called a Folner sequence, is easily described for abelian groups, and well-understood for some classes of solvable groups (\cite{PSC}, \cite{Ers}). Many non-solvable amenable groups are directed groups acting on rooted trees. This family of groups gathers many examples with "exotic" properties, such as infinite torsion groups of intermediate growth constructed by Aleshin (\cite{Ale}, \cite{Gri}) or groups with non-uniform exponential growth by Wilson \cite{Wil1}. 

Their amenability in the case of bounded valency was shown in \cite{Bri1} by use of Kesten's probabilistic criterion \cite{Kes}. The strategy, introduced by Bartholdi and Virag in \cite{BV}, is to show that a self-similar random walk on a Cayley graph diffuses slowly, in the sense that its return probability does not decay exponentially, or that its entropy is sublinear (\cite{KV}). The same method permits to show that automata groups are amenable when their activity is bounded \cite{BKN} or linear \cite{AAV}. Though it ensures their existence, such a probabilistic proof does not exhibit Folner sets. 

For the groups of \cite{Ale} and \cite{Gri}, subexponential growth easily implies the existence of a subsequence of the family of balls (for a word length) which is a Folner sequence, but it is not known if the whole sequence of balls is Folner and the subsequence (even though it has density $1$) is not explicit. Even for groups of polynomial growth, it is not elementary to show that balls form a Folner sequence, a result due to Pansu \cite{Pan}, using technics from Gromov \cite{Gro}.

The object of the present article is to exhibit explicit Folner sets for some groups with a property denoted $\Dc\Pc$, containing in particular directed groups acting on a rooted tree by alternate permutations. A group $\G$ with property $\Dc\Pc$ is defined (see section \ref{gen}) by two subgroups $A$ finite and $H$ finitely generated, together with an action on a rooted tree with valency sequence $(d_k)_{k \in \N}$. The main result is:
\begin{theorem}
Let $\G$ have property $\Dc\Pc$ with $H$ amenable and $\frac{d_k}{\log k} \rightarrow 0$, then the group $\G$ is amenable.
\end{theorem}
In particular, the direct description of Folner sets provides a new proof, using neither random walks nor word length, that directed groups acting on a rooted tree of bounded valency are amenable (\cite{Bri1}). It also provides many new examples of amenable directed groups acting on a tree of unbounded sublogarithmic valency. Moreover it permits to reprove amenability of automata groups with bounded activity by methods different from \cite{BKN}. 

The article is structured as follows. Rooted trees and their automorphism groups are described in section \ref{rootree}. Section \ref{first} is devoted to the construction of explicit Folner sets for the archetypal example of the alternate mother group $G_d$ acting on a regular rooted tree of valency $d\geq 5$. This example, treated first for simplicity of notations, is generalized to groups with property $\Dc\Pc$ in section \ref{gen}. Finally, section \ref{examples} is devoted to the construction of groups satisfying $\Dc\Pc$, including the saturated alternate directed groups, and some groups acting on trees with unbounded valency.

\section{Rooted trees and their groups of automorphisms} \label{rootree}

Let $S_d$ denote the group of permutations of the set $\{1,\dots,d\}$ with $d$ elements, and  $\Ac_d=\Ac_{\{1,\dots,d\}}$ denote the subgroup of alternate permutations. 

Given a sequence $\bar d=(d_k)_{k\geq 0}$ of integers $\geq 2$, the spherically homogeneous rooted tree $T_{\bar d}$ is the graph with vertex set $\{t_0t_1\dots t_k|t_i \in \{1,\dots,d_i\},k \geq -1\} $, including the empty sequence $\emptyset$, called the root, corresponding to $k=-1$, and edge set $\{(t_0\dots t_k,t_0 \dots t_k t_{k+1})\}$. The vertex set restricted to a fixed $k$ is called the $(k+1)$st level of the tree. It is the direct product $\{1,\dots,d_0\} \times \dots \times \{1,\dots,d_k\}$. When the sequence $\bar d$ is constant equal to $d$, the tree is called $d$-regular, denoted $T_d$.

The group of automorphisms $Aut(T_{\bar d})$ of the rooted tree $T_{\bar d}$ is the group of graph automorphims that fix the root $\emptyset$. It satisfies a canonical isomorphism:
\begin{eqnarray}\label{iso}\f: Aut(T_{\bar d}) \xrightarrow{\simeq} Aut(T_{s \bar d}) \wr S_{d_0}, \end{eqnarray}
where $
s \bar d=(d_k)_{k\geq 1}$ is the shifted sequence obtained by deleting the first entry, and $G\wr S_d=(G\times \dots \times G) \rtimes S_d$ is the semi-direct product where $S_d$ acts by permuting factors, called wreath product. Since $\f$ is canonical, we identify $g$ and $\f(g)$ and write $g=(g_1,\dots,g_{d_0})\s=(g_{t_0})\s$. The product rule is $gg'=(g_1g'_{\s(1)},\dots,g_{d_0}g'_{\s(d_0)})\s\s'$, where $g$ is applied before $g'$.

By iterating the wreath product ismorphism (\ref{iso}), a family of canonical isomorphisms is obtained:
\begin{eqnarray} \label{isok} Aut(T_{\bar d}) \simeq Aut(T_{s^k \bar d}) \wr S_{d_{k-1}} \wr \dots \wr S_{d_0}.\end{eqnarray}
Identifications are denoted $g=(g_{t_0\dots t_{k}})(\s_{t_0\dots t_{k-1}})\dots(\s_{t_0})\s$, where $(\s_{t_0\dots t_j})$ is  a sequence of permutations in $S_{d_j}$ indexed by the $(j+1)$st level of the tree and $(g_{t_0\dots t_{k}})$ is a sequence of automorphisms of $T_{\s^k \bar d}$ indexed by level $k+1$. The automorphism $g$ is determined by the whole sequence of permutations $(\s_v)_{v \in T_{\bar d}}$, called its portrait. 

The automorphism $g$ is said to be alternate if all the permutations $\s_v$ of its portrait are alternate permutations. Denote $Aut^{alt}(T_{\bar d})$ the group of alternate automorphisms of $T_{\bar d}$. It also satisfies canonical isomorphisms:
$$Aut^{alt}(T_{\bar d}) \simeq Aut^{alt}(T_{s^k \bar d}) \wr \Ac_{d_{k-1}} \wr \dots \wr \Ac_{d_0}. $$

The neutral element of a group $G$ is denoted $e_G$ or $e$. 

\section{Folner sets of the alternate mother group}\label{first}

\subsection{The alternate mother group} In the case of a $d$-regular rooted tree $T_d$, the canonical wreath product isomorphism of the group of alternate automorphisms has the form:
\begin{eqnarray}\label{isod}\f:Aut^{alt}(T_d) \xrightarrow{\simeq} Aut^{alt}(T_d) \wr \Ac_d. \end{eqnarray}
It permits to define recursively some alternate automorphisms of $T_d$ as follows.

\begin{itemize}
\renewcommand{\labelitemi}{$\bullet$}

\item Given $\s$ in $\Ac_d$, denote $A=\{ a(\s)|\s \in \Ac_d\}\simeq \Ac_d$ with:
$$\f(a(\s))=(e,\dots,e)\s. $$
The elements $a=a(\s)$ of $A$ are alternate automorphisms of $T_d$, called rooted automorphisms, because the portrait of $a(\s)$ is given by $\s_\emptyset=\s$ and $\s_v=e$ for $v \neq \emptyset$.

\item Given $a_2,\dots,a_d$ in $\Ac_d$ and $\r$ in $Fix_{\Ac_d}(1)=\Ac_{\{2,\dots,d\}}=\Ac_{d-1}$,  the alternate automorphism $b=b(a_2,\dots,a_d,\r)$ satisfies under the wreath product isomorphism:
\begin{eqnarray*}\f(b(a_2,\dots,a_d,\r))=(b(a_2,\dots,a_d,\r),a_2,\dots,a_d)\r. \end{eqnarray*}

This defines recursively a tree automorphism $b=b(a_2,\dots,a_d,\r)$  with portrait the family of permutations $(\s_v)_{v \in T_{d}}$ given by $\s_{1\dots 1 1}=\r$, $\s_{1\dots 1t}=a_t$ for $2 \leq t \leq d$ and $\s_v=e$ for the other vertices $v$.

Denote $B=\{b(a_2,\dots,a_d,\r)|a_2,\dots,a_d \in \Ac_d,\r \in Fix_{\Ac_d}(1)\}$. The elements of $B$ are called directed. The set $B$ forms a finite subgroup of $Aut(T_d)$. Indeed, the following is an isomorphism:
\begin{eqnarray}\label{Bfinite}\begin{array}{rll} B & \rightarrow  & (\Ac_d \times \dots \times \Ac_d) \rtimes \Ac_{\{2,\dots,d\}} \\ b(a_2,\dots,a_d,\r) & \mapsto & (a_2,\dots,a_d)\r. \end{array}. \end{eqnarray}

\item The alternate mother group $G_d$ is the subgroup of alternate automorphisms of $T_d$ generated by the sets $A,B$:
$$G_d=\langle A,B \rangle <Aut^{alt}(T_d). $$
\end{itemize}

By construction, the group $G_d$ is an automata group. It is essentially the mother group of degree 0 (see \cite{BKN}, \cite{AAV}), but the permutations involved are alternate. Since $\Ac_d$ is simple hence perfect for $d\geq 5$, the group $G_d$ satisfies the:

\begin{proposition}\label{wr}
If $d\geq 5$, the canonical isomorphism (\ref{isod}) induces an isomorphism:
$$\f: G_d \xrightarrow{\simeq} G_d \wr \Ac_d. $$
\end{proposition}

This isomorphism will also be considered canonical $ G_d \simeq G_d \wr \Ac_d$, and the elements $g$ and $\f(g)$ will be identified in the remainder of this section. The proposition follows from the:
 
\begin{fact}\label{perfect1}
Let $d \geq 5$, then for any generator $a=a(\s) \in A$ and $b=b(a_2,\dots,a_d,\r) \in B$, the elements $(a,e,\dots,e)$ and $(b,e,\dots,e)$ belong to $\f(G_d)$.
\end{fact}

Recall the conjugacy notation $g^a=aga^{-1}$, and observe that for $g=(g_1,\dots,g_d)\s$ and $a$ in $A$, one has $g^a=(g_{a(1)},\dots,g_{a(d)})\s^a$.

\begin{proof}[Proof of Fact \ref{perfect1}]
Take $\t$ in $\Ac_d$ such that $\t(1)=1$ and $\t^{-1}(2)=3$ and observe the commutator relations:
\begin{eqnarray*}
\f(b)&=&\f(b(\a_2,e_A,\dots,e_A,e_A))= (b,\a_2,e_G,e_G,\dots,e_G)e_A, \\
\f(b'^{\t})&=&\f(b(\a'_2,e_A,\dots,e_A,e_A)^{\t}) = (b',e_G,\a'_2,e_G,\dots,e_G)e_A, \\
~\f([b,b'^{\t}]) &=& ([b,b'],e_G,e_G,e_G,\dots,e_G). 
\end{eqnarray*}
As $[b(\a_2,e_A,\dots,e_A,e_A),b(\a_2',e_A,\dots,e_A,e_A)]=b([\a_2,\a_2'],e_A,\dots,e_A,e_A)$ and as the group $A\simeq\Ac_d$ is perfect (because it is simple), any element $a_2$ in $A\simeq \Ac_{d}$ is a product of commutators. This shows that $\f(G_d)$ contains $(b_2,e_G,\dots,e_G)$ for any $b_2=b(a_2,e_A,\dots,e_A,e_A)$ with $a_2$ in $\Ac_d$. Moreover for any $b_\emptyset=b(e_A,\dots,e_A,\r)$ with $\r$ in $Fix_A(1)\simeq \Ac_{d-1}$, the group $\f(G_d)$ contains $\f(b_\emptyset a(\r^{-1}))=(b_\emptyset,e_G,\dots,e_G)$. 

Now the elements $b_2=b(a_2,e_A,\dots,e_A,e_A)$ and $b_\emptyset=b(e_A,\dots,e_A,\r)$ generate $B$ by isomorphism (\ref{Bfinite}), because $\r$ in $\Ac_{\{2,\dots,d\}}$ and $(a_2,e_A,\dots,e_A)$ for $a_2$ in $\Ac_d$ generate the finite group $(\Ac_d \times \dots \times \Ac_d) \rtimes \Ac_{\{2,\dots,d\}}$. Thus $\f(G_d)$ contains $(b,e_G,\dots,e_G)$ for any $b$ in $B$.

Finally given $a_2$ in $A$, for $b_2=b(a_2,e_A,\dots,e_A,e_A)=(b_2,a_2,e_G,\dots,e_G)$, the element $(b_2^{-1},e_G,\dots,e_G)$ belongs to $\f(G_d)$ by the above. So do $(b_2^{-1},e_G,\dots,e_G)\f(b_2)=(e_G,a_2,e_G,\dots,e_G)$ and $(e_G,a_2,e_G,\dots,e_G)^\t=(a_2,e_G,\dots,e_G)$ for $\t$ in $\f(A) \simeq A \simeq \Ac_d$ such that $\t^{-1}(2)=1$.
\end{proof}
 
\begin{proof}[Proof of Proposition \ref{wr}]
By definition of the generators of $G_d$, the morphism $\f$ is an embedding into the wreath product $G_d \wr \Ac_d$. The key point is that this embedding is surjective. Clearly $\f(A) \simeq A \simeq \Ac_d$ is the set of rooted automorphisms. Moreover, Fact \ref{perfect1} shows that $G_d \times \{e\} \times \dots \times \{e\}$ is in $\f(G_d)$. As $\Ac_d$ acts transitively on $\{1,\dots,d\}$, conjugation shows that $\{e\}\times \dots \times G_d \times \dots \times \{e\}$ also belongs to $\f(G_d)$ for any position of the non-trivial factor. Then $G_d \times \dots \times G_d$ belongs to $\f(G_d)$ by product. This proves the wreath product isomorphism.
\end{proof}

\subsection{Definition of Folner sets}

For a group $\G$ with finite generating set $S$, the boundary of a subset $L \subset \G$ is defined as:
$$\partial_S L=\{\g \in L| \exists s \in S, \g s \notin L\}. $$
The interior of $L$ is the set $Int_S(L)=L\setminus \partial_S L$.

A sequence $L_k$ of subsets of $\G$ is a Folner sequence if $\frac{|\partial L_k|}{|L_k|}\rightarrow 0$. By \cite{Fol}, a finitely generated group $\G$ is amenable if and only if it admits a Folner sequence for some (equivalently for any) finite generating set $S$. For the remainder of this section, the set $S=A \cup B$ is considered the canonical generating set of $G_d$ and the notations $\partial L$ and $Int(L)$ stand for $\partial_{A\cup B}L$ and $Int_{A\cup B}$ respectively.

Let us define a sequence of subsets of $G_d$ as follows:
$$L_0=\{g \in G_d|\exists \b \in B, \a_2,\dots,\a_d,\s \in A, g=(\b,\a_2,\dots,\a_d)\s\}.$$
By induction on $k$, define:
$$L_{k+1}=\left\{g\in G_d \left| \begin{array}{r} g=(g_1,\dots,g_d)\s , \textrm{ such that }  \s \in A,\forall t \in \{1,\dots,d \}, g_t \in L_k, \\  \textrm{ and } \exists T\in \{1,\dots,d \},g_T \in Int(L_k)\end{array}\right.\right\}. $$
By Proposition \ref{wr}, the sets $L_k$ are included in $G_d$ for $d\geq 5$, and not just in the automorphism group $Aut(T_d)$.

\begin{theorem}\label{deg0}
For $d\geq 5$, the sets $L_k$ form a Folner sequence for $G_d$. In particular, the group $G_d$ is amenable.
\end{theorem}

The group $G_d$ was known to be amenable by \cite{Bri1} (use of Kesten criterion on return probability) or \cite{BKN} (triviality of the Poisson boundary). However, these proofs, based on contraction in the wreath product of word length for some random walks, did not provide explicit Folner sets. The following proof uses neither random walks, nor word length.

\begin{remark}\label{isop} Estimation on the rate of convergence of $|\partial L_k|/|L_k|$ to zero and on the cardinality of $L_k$ will show that the Folner function $Fol(n)=\min\{|L| /n|\partial L|\leq |L| \}$ is bounded above by a function $\exp_C(\exp_C(n^{d-1+\e}))$ for some constant $C$ and any positive $\e$ where $\exp_C(x)=C^x$. However, the return probability of a symmetric simple random walk on $G_d$ is bounded below by $\exp(-n^{\b_d})$ for $\b_d=\log d/\log \frac{d^2}{d-1}$ by Theorem 6.1 in \cite{Bri3}. Combined with Corollary V.2 in \cite{Gr}, this shows that the function $Fol(n)$ cannot be bounded below by $\exp(n^\alpha)$ for $\a=2 \b_d/1-\b_d$. Thus the sequence $L_k$ exhibited here is far from optimal as a Folner sequence.
\end{remark}

\subsection{Proof of Theorem \ref{deg0}}

Observe that for any $a$ in $A$ and $g=(\b,\a_2,\dots,\a_d)\s$ in $L_0$, the element $ga=(\b,\a_2,\dots,\a_d)\s a$ still belongs to $L_0$. Moreover, for any $b=b(a_2,\dots,a_d,\r)=(b,a_2,\dots,a_d)\r$ in $B$, one has:
$$gb=\left\{ \begin{array}{ll} (\b b,\a_2 a_{\s(2)},\dots,\a_d a_{\s(d)})\s\r & \textrm{ if }\s^{-1}(1)= 1,  \\
(\b a_{\s(1)},\a_2 a_{\s(2)},\dots,\a_{\s^{-1}(1)}b,\dots,\a_d a_{\s(d)})\s\r & \textrm{ if }\s^{-1}(1) \neq 1.  \end{array} \right. $$
As the sets $A$ and $B$ are finite groups, this shows equivalence of (1), (2) and (3) in the:
\begin{fact}\label{factL0}
The following are equivalent:
\begin{enumerate}
\item $g$ belongs to $Int(L_0)$,
\item $gb \in L_0$ for all $b \in B$,
\item $\s^{-1}(1)=1$,
\item $gb \in Int(L_0)$ for all $b \in B$.
\end{enumerate}
In particular, $\frac{|Int(L_0)|}{|L_0|}=\frac{1}{d}$, hence $\d_0=\frac{|\partial L_0|}{|L_0|}=1-\frac{1}{d}$.
\end{fact}

\begin{proof} Point (4) is equivalent to (3) due to the fixed point assumption $\r(1)=1$ in the definition of $B$, which guarantees that $(\s\r)^{-1}(1)=(\r^{-1}\s^{-1})(1)=\s^{-1}(\r^{-1}(1))=1$ when $\s^{-1}(1)=1$.

The evaluation of $\d_0$ is done by counting $|L_0|=|B||A|^d$ as $g$ is described by $\b,\a_2,\dots,\a_d,\s$, and condition $\s^{-1}(1)=1$ occurs with probability $\frac{1}{d}$.
\end{proof}

\begin{lemma}\label{lemmaLk}
Let $g \in L_k$, the following are equivalent:
\begin{enumerate}
\item $g$ belongs to $Int(L_k)$,
\item $gb \in L_k$ for all $b \in B$,
\item $\s^{-1}(1)\in I(g)=\{T|g_T \in Int(L_{k-1})\}$,
\item $gb \in Int(L_k)$ for all $b \in B$.
\end{enumerate}
\end{lemma}

\begin{proof}[Proof of Lemma \ref{lemmaLk}]
The case $k=0$ is treated by Fact \ref{factL0} with convention that $I(g)=\{1\}$ if $g \in L_0$. Assume by induction that the result is true for $k-1$, and prove it for $k$. 

Again $ga=(g_1,\dots,g_d)\s a$ belongs to $L_k$ for any value of $a$ in $A$, $g$ in $L_k$. Moreover:
$$gb=(g_1a_{\s(1)},\dots,g_{\s^{-1}(1)}b,\dots,g_d a_{\s(d)})\s\r. $$

Suppose (3) holds true, that is $g_{\s^{-1}(1)}\in  Int(L_{k-1})$, then as (1) implies (4) for $k-1$, the element $g_{\s^{-1}(1)}b$ belongs to $Int(L_{k-1})$ for any $b$ in $B$, so that $gb$ belongs to $L_k$ for any $b$ in $B$, proving (2). Then (1) follows because $ga$ also belongs to $L_k$ for $a$ in $A$, hence $g$ is an interior point of $L_k$.

Suppose (3) does not hold, so $g_{\s^{-1}(1)}\in  \partial L_{k-1}$. By equivalence of (1) and (2) for $k-1$, there exists $b$ in $B$ such that $g_{\s^{-1}(1)}b\notin L_{k-1}$, so that $gb$ is not in $L_k$, disclaiming (1) and (2) for $g$. This proves equivalence of (1), (2) and (3) for $k$.

Now $gb$ belongs to $Int(L_k)$ if and only if $(\s\r)^{-1}(1) \in I(g)$ by equivalence of (1) and (3). But $(\s\r)^{-1}(1)=\s^{-1}(\r^{-1}(1))=\s^{-1}(1)$ because $\r(1)=1$. So (3) implies (4). Obviously, (4) implies (2), closing step $k$ of induction.
\end{proof}

There remains to evaluate the sizes of the interior and boundary of $L_k$. Set:
$$\d_k=\frac{|\partial L_k|}{|L_k|}, \quad 1-\d_k=\frac{|Int(L_k)|}{|L_k|}. $$

\begin{lemma}\label{delta}The sequence $(\d_k)$ satisfies:
$$1-\d_{k+1}=\frac{1-\d_k}{1-\d_k^d}. $$
\end{lemma}

\begin{proof}[Proof of Lemma \ref{delta}]
Given a subset $I \subset \{1,\dots,d\}$, denote:
$$J_I=\{g=(g_1,\dots,g_d)\s|\forall T \in I,g_T \in Int(L_k) \textrm{ and }\forall t \notin I, g_t \in \partial L_k \}. $$
By definition, $L_{k+1}$ is the disjoint union $L_{k+1}=\sqcup_{|I| \geq 1}J_I$.

For $i=|I|$, the size of $J_I$ and its intersection with $Int(L_{k+1})$ are evaluated as:
\begin{eqnarray*}
|J_I| &=& |\Ac_d||Int(L_k)|^i|\partial L_k|^{d-i}=|\Ac_d||L_k|^d(1-\d_k)^i\d_k^{d-i}, \\
|J_I\cap Int(L_{k+1})| &=& \frac{|I|}{d} |\Ac_d||Int(L_k)|^i|\partial L_k|^{d-i}=\frac{i}{d}|J_I|,
\end{eqnarray*}
where the factor $\frac{i}{d}$ comes from (3) of Lemma \ref{lemmaLk}. Denote $C_d^i$ the number of subsets of size $i$ in $\{1,\dots,d\}$, and use the mean of binomial distribution to get:
\begin{eqnarray*}
|Int(L_{k+1})|&=& \sum_{i=1}^d C_d^i (1-\d_k)^i\d_k^{d-i}\frac{i}{d}|L_k|^d|\Ac_d|=(1-\d_k)|L_k|^d|\Ac_d|, \\
|L_{k+1}| &=& \sum_{i=1}^d C_d^i (1-\d_k)^i\d_k^{d-i}|L_k|^d|\Ac_d|=(1-\d_k^d)|L_k|^d|\Ac_d|.
\end{eqnarray*}
This shows that:
$$1-\d_{k+1}=\frac{|Int(L_{k+1})|}{|L_{k+1}|}=\frac{1-\d_k}{1-\d_k^d}. $$
\end{proof}

\begin{proof}[Proof of Theorem \ref{deg0}]
As $\d_k>0$, Lemma \ref{delta} implies $1-\d_{k+1}>1-\d_k$, so the sequence $(\d_k)$ is decreasing, tending to a limit $\d$ satisfying $1-\d=\frac{1-\d}{1-\d^d}$, hence $\d$ is $0$ (or $1$, ruled out by $\d_0<1$).
\end{proof}

More precisely, Lemma \ref{delta} implies that for any $\eta <\frac{1}{d-1}$, one has $\d_k =O(k^{-\eta})$, as shown below in Lemma \ref{epsilon}. On the other hand: 
$$|L_k| =|B|^{d^k}|A|^{(d-1)d^k+(d^k+\dots+d+1)} \geq 2^{2^k}.$$
This provides the estimate on the Folner function in remark \ref{isop}.

\begin{remark}\label{remcomb}
Lemma \ref{lemmaLk} provides a complete combinatorial description of $L_k$. An element $g$ of $G_d$ has the form $g=(g_{t_0\dots t_{k}})(\s_{t_0\dots t_{k-1}})\dots(\s_{t_0})\s$ in the $k$th iteration of the wreath product. Such an element $g$ belongs to $L_k$ if and only if it satisfies the three following conditions:
\begin{enumerate}
\item $\forall t_0\dots t_{k-1}$, the element $g_{t_0\dots t_{k-1}1}$ is in $B$ and $g_{t_0\dots t_{k-1}2},\dots,g_{t_0\dots t_{k-1}d}$ are in $A$,
\item $\forall t_0\dots t_{k-2}$, the set $I(t_0\dots t_{k-2})=\{T_{k-1}| \s_{t_0\dots t_{k-2}T_{k-1}}^{-1}(1)=1\}$
is non-empty.
\item $\forall 3 \leq l \leq k+1, \forall t_0 \dots t_{k-l}$, the set 
$$I(t_0\dots t_{k-l})=\{T_{k-l+1}| \s_{t_1\dots t_{k-l}T_{k-l+1}}^{-1}(1) \in I(t_1\dots t_{k-l}T_{k-l+1}) \},$$ defined by induction on $l$, is non-empty (for $l=k+1$, consider $I(\emptyset)$ where $\emptyset$ is the root vertex of $T_d$). 
\newline 
\newline
The element $g$ belongs to $Int(L_k)$ if and only if it satisfies (1), (2), (3) and moreover:
\item $\s^{-1}(1) \in I(\emptyset)=\{T|\s_{T} \in I(T)\}$.
\end{enumerate}
Note that condition (2) is a specific case of condition (3) where $I(t_0\dots t_{k-1})=\{1\}$ for all $t_0\dots t_{k}$.
As an interpretation, say a vertex $v=t_0\dots t_l$ with $l\leq k-1$ is open if $\s_v^{-1}(1) \in I(v)$. Conditions (1), (2), (3) ensure that $g$ belongs to $L_k$ if and only if each vertex $v$ has at least one neighbour of next level $vT$ which is open. Condition (4) ensures that $g$ is in the interior $Int(L_k)$ if and only if the root itself is open.
\end{remark}

\section{Generalization}\label{gen}

\subsection{Property $\Dc\Pc$} Theorem \ref{deg0} can be generalized to the following wider setting. 

\begin{definition}\label{defdp}
A sequence of groups  is said to have property $\Dc\Pc$ if it satisfies the two following conditions for all $i$ in $\N$:
\begin{enumerate}
\item the group $\G_i$ contains two subgroups $A_i$ and $H_i$ such that:
\begin{enumerate}
\item the set $A_i \cup H_i$ generates the group $\G_i$,
\item the group $A_i$ is finite, acting transitively on a finite set $\{1,\dots,d_i\}$ of size $d_i \geq 2$,
\item the group $H_i$ is finitely generated, 
\end{enumerate}
\item there is an isomorphism:
$$ \f_i: \G_i \longrightarrow \G_{i+1} \wr A_i = (\G_{i+1} \times \dots \times \G_{i+1}) \rtimes A_i, $$
with $d_i$ factors in the direct product, on which $A_i$ is acting by permutation of coordinates, according to its transitive action on $\{1,\dots,d_i\}$. Moreover, this isomorphism $\f_i$ satisfies:
\begin{enumerate}
\item $\forall s \in A_i, \f_i(s)=(e_{\G_{i+1}},\dots,e_{\G_{i+1}})s$,
\item $\forall h_i \in H_i, \exists h_{i+1} \in H_{i+1}, \exists a_2,\dots,a_{d_0} \in A_{i+1}, \exists \r \in A_i$, with $\r(1)=1$ and:
$$\f_i(h_i)=(h_{i+1},a_2,\dots,a_{d_i})\r, $$
where the groups $A_{i+1}$ and $H_{i+1}$ are the subgroups of $\G_{i+1}$ satisfying condition (1).
\end{enumerate}
\end{enumerate}
A group $\G$ is said to have property $\Dc\Pc$ if there exists a sequence $\{\G_i\}_{i \in \N}$ with property $\Dc\Pc$ such that $\G \simeq \G_0$.
\end{definition}


Groups with property $\Dc\Pc$ are related to the groups of non-uniform growth constructed by Wilson (see \cite{Wil1},\cite{Wil2},\cite{Bri1}). In particular, if all the groups $\G_i$ of a sequence with property $\Dc\Pc$  are generated by a finite number (independent of $i$) of involutions, and if all the groups $A_i$ involved are alternate groups $\Ac_{d_i}$ acting on sets of size $d_i\geq 29$, then they have non-uniform growth by \cite{Wil2}. This is the case of the examples  in proposition \ref{dpbounded} below.

\begin{fact}\label{tree}
If $\G_0$ belongs to $\Dc\Pc$, there exists a sequence $\bar d=(d_i)_i$ of integers $d_i \geq 2$, and the group $\G_0$ is acting by automorphisms on the spherically homogeneous rooted tree $T_{\bar d}$. This action is transitive on each level.
\end{fact}

Note that this action on the tree is not necessarily faithful. For instance, the subgroup $F$ of the group $\G=\G(\Ac_{d_0},A_{\bar d},F)$ of section 2.4 of \cite{Bri3} has a trivial action on the tree $T_{\bar d}$, even though $\G$ has property $\Dc\Pc$, for the sequence $\G_i=\G(\Ac_{d_i},A_{s^i\bar d},F)$.

\begin{proof}
Combining the isomorphisms of Definition \ref{defdp}, there is an isomorphism $\G_0 \simeq \G_{i+1} \wr A_i \wr \dots \wr A_0$. As $A_i$ is acting transitively on $\{1,\dots,d_i\}$, the group $A_i \wr \dots \wr A_0$ is acting transitively on $\{1,\dots,d_0\} \times \dots \times \{1,\dots,d_i\}$, which is the $i+1$st level of $T_{\bar d}$. Taking the limit with $i$, this provides the action on the tree $T_{\bar d}$.
\end{proof}

Consider a group $\G_0$ with property $\Dc\Pc$ and take notations of Definition \ref{defdp}. Let $B_0$ be a fixed finite generating set of $H_0$. Define inductively the sequence $B_i$ of subsets of $H_i$ by condition (2)(b). For any $b_i \in B_i$, set $b_{i+1}$ to be the unique element in $H_{i+1}$ such that $\f_i(b_i)=(b_{i+1},a_2,\dots,a_{d_i})\r$, and $B_{i+1}=\{b_{i+1}|b_i \in B_i\}$. By construction, $B_i$ is a subset of $H_i$ of size $\leq |B_0|$. It is not true in general that $B_i$ generates $H_i$ for all $i\in \N$, however, we have:

\begin{fact}
Let $\G_0$ have property $\Dc\Pc$ and $B_i$ be as above. Then the conditions of Definition \ref{defdp} are fulfilled if $H_i$ is replaced by the subgroup $\langle B_i \rangle$ of $\G_i$.
\end{fact}

\begin{proof}
It is sufficient to check conditions (1) and (2) of Definition \ref{defdp} for all $i$ in $\N$. For $i=0$, (1) is true since $H_0=\langle B_0 \rangle$, and (2) is true by definition of $B_1$. Then:
$$\G_1 \wr A_0 \simeq \f_0(\G_0) \subset \langle B_1 \cup A_1 \rangle \wr A_0. $$
The inclusion is forced to be  an equality since $A_1, B_1$ are included in $\G_1$, thus $A_1 \cup B_1$ generates $\G_1$. This shows that $H_1$ can be replaced by $\langle B_1\rangle$. The fact follows by induction.
\end{proof}

This shows that up to replacing the groups $H_i$ by the groups $\langle B_i \rangle$, which does not affect the groups $\G_i$, we may and shall assume that $B_i$ is a canonical generating set for $H_i$. 

\begin{fact}\label{fact2}
The group $H_0$ is amenable if and only if the groups $H_i$ are amenable for all $i$.
\end{fact}

\begin{proof}
By (2)(b), the restriction of $\f_0$ to $H_0$ provides an embedding:
$$\f_0|_{H_0}:H_0 \hookrightarrow H_1 \times (A_1 \wr Fix_{A_0}(1)). $$
As the second factor is a finite group, amenability of $H_1$ implies that of $H_0$.

Conversely assume that $H_0$ is amenable. By (2)(b), any relation between the generators in $B_0$ implies a relation between the corresponding generators in $B_1$ of $H_1$. Thus $H_1$ is a quotient of $H_0$, hence is amenable.

The same proof shows that amenability of $H_{i+1}$ is equivalent to that of $H_i$.
\end{proof} 

\begin{question}
If a group $\G_0=\langle A_0 \cup H_0\rangle$ has property $\Dc\Pc$ with $H_0$ amenable, is the group $\G_0$ amenable?
\end{question}

The following theorem provides a partial answer, with a condition on the sequence of integers $\bar d=(d_i)_i$.

\begin{theorem}\label{MT}
Let $\G_0$ have property $\Dc\Pc$ with $H_0$ amenable and $\bar d$ growing sufficiently slowly (for instance $\frac{d_k}{\log k}\rightarrow 0$), then $\G_0$ is amenable.
\end{theorem}

This theorem generalizes Theorem \ref{deg0}. The proof is similar, though slightly more technical. 

\subsection{Proof of Theorem \ref{MT}}

Given $\G_0=\langle A_0 \cup B_0\rangle$ with property $\Dc\Pc$, consider the associated sequence of finitely generated groups $\G_K=\langle A_K \cup B_K\rangle$, where $B_K$ is the canonical generating set of the group $H_K$. The notions of interior and boundary used below refer to these generating sets.

To ease notations, write $g$ instead of $\f_K(g)$.  For $\O \subset H_{K+1}$, set:
$$L_0^K(\O)=\{g \in \G_K|\exists h \in \O,\a_2,\dots,\a_{d_K} \in A_{K+1},\s \in A_K, g=(h,\a_2,\dots,\a_{d_K})\s\}, $$
$$\i L_0^K(\O)=\{g \in L_0^K(\O)|\s^{-1}(1)=1\}, $$
and by induction for $1 \leq k \leq K$, set:
$$ L_k^K(\O)=\{g=(g_1,\dots,g_{d_{K-k}})\s\in\G_{K-k}| \forall t, g_t \in L_{k-1}^K(\O),\exists T, g_T \in \i L_{k_1}^K(\O) \},$$
$$\i L_k^K(\O)=\{g \in L_k^K(\O)|g_{\s^{-1}(1)} \in \i L_{k-1}^K(\O)\}. $$

The sets $\i L_k^K(\O)$ should be considered as ``combinatorial interiors'' of $L_k^K(\O)$. They satisfy a combinatorial description as Remark \ref{remcomb}, but slightly differ from the actual interior of $L_k^K(\O)$, unless the set $\O$ has empty boundary (see Remark \ref{combin} below). Fact \ref{factL0} generalizes as:

\begin{fact}\label{L0K}
The three following are equivalent:
\begin{enumerate}
\item $g \in Int(L_0^K(\O))$,
\item $gb_K \in L_0^K(\O)$ for all $b_K \in B_K$,
\item $\s^{-1}(1)=1$ and $h \in Int(\O)\subset \O \subset H_{K+1}$.
\newline \newline
Moreover they also imply:
\item $gb_K \in \i L_0^K(\O)$ for all $b_K \in B_K$.
\end{enumerate}
In particular, $\frac{|Int(L_0^K(\O))|}{|L_0^K(\O)|}=\frac{|Int(\O)|}{d_K|\O|}$, and $\d_0^K(\O)=\frac{|\partial L_0^K(\O)|}{|L_0^K(\O)|}=1-\frac{|Int(\O)|}{d_K|\O|}$.
\end{fact}

\begin{proof} 
Let $g=(h,\a_2,\dots,\a_{d_K})\s$ belong to $L_0^K(\O)$. By (2)(a) of Definition \ref{defdp} for $A_K$, the element $ga_K$ still belongs to $L_0^K(\O)$ for $a_K$ in $A_K$. This proves equivalence of (1) and (2).

Now take $b_K=(b_{K+1},a_2,\dots,a_{d_K})\r$ in $B_K$, then:
$$gb_K=\left\{ \begin{array}{ll} (h b_{K+1},\a_2 a_{\s(2)},\dots,\a_d a_{\s(d)})\s\r & \textrm{ if }\s^{-1}(1)= 1,  \\
(h a_{\s(1)},\a_2 a_{\s(2)},\dots,\a_{\s^{-1}(1)}b_{K+1},\dots,\a_d a_{\s(d)})\s\r & \textrm{ if }\s^{-1}(1) \neq 1.  \end{array} \right. $$
This shows that $gb_K$ belongs to $L_0^K(\O)$ for all $b_K$ if and only if $\s^{-1}(1)=1$ and $h$ belongs to $Int(\O)$, i.e. equivalence of (2) and (3).

This implies (4) because then $(\s\r)^{-1}(1)=1$. Computing the sizes follows from~(3).
\end{proof}


\begin{notation}Let $g =(g_1,\dots,g_{d_i})\s=(g_{t_i})\s$ in $\G_i$, with $\s$ in $A_i$, $g_{t_i}$ in $\G_{i+1}$ for $t_i \in \{1,\dots,d_i\}$ by identification of $g$ with $\f_i(g)$. More generally, identify $g_{t_i\dots t_j}$ with $\f_{j+1}(g_{t_i\dots t_j})$ for $i\leq j \leq K$ and denote:
$$g=(g_{t_i\dots t_K})(\s_{t_i\dots t_{K-1}})\dots(\s_{t_i})\s,$$
where $\s_{t_i\dots t_{j}}$ belongs to $A_{j+1}$ and $g_{t_i\dots t_K}$ to $\G_{K+1}$. Set $\t_i=\s^{-1}(1)\in \{1,\dots,d_i\}$, and by induction $\t_{j+1}=(\s_{\t_i\dots \t_j})^{-1}(1)\in \{1,\dots,d_{j+1}\}$, which guarantees $g(\t_i\t_{i+1}\dots \t_j)=11\dots 1$ for the action on the tree of fact \ref{tree}. 
\end{notation}

The following generalizes Lemma \ref{lemmaLk}.

\begin{lemma}
For $0\leq k \leq K$, the three following are equivalent:
\begin{enumerate}
\item $g \in Int(L_k^K(\O))$,
\item $gb_{K-k} \in L_k^K(\O)$ for all $b_{K-k} \in B_{K-k}$,
\item $g \in \i L_k^K(\O)$ (i.e. $\s^{-1}(1) \in I(g)=\{T|g_T \in \i L_{k-1}^K(\O)\}$) and $g_{\t_{K-k}\dots \t_K} \in Int(\O)$.
\newline \newline
Moreover, they also imply:
\item $gb_{K-k} \in \i L_k^K(\O)$ for all $b_{K-k} \in B_{K-k}$.
\end{enumerate}
\end{lemma} 

Observe that if $g \in \i L_k^K(\O)$, then $g_{\t_{K-k}\dots \t_K} \in \O$, by definitions of $\i L_k^K(\O)$ and $\t_{K-k}\dots \t_K$.

\begin{proof}
Let $g=(g_1,\dots,g_{d_{K-k}})\s$ belong to $L_k^K(\O)$. For $a$ in $A_{K-k}$, $ga$ still belongs to $L_k^K(\O)$ (no condition on $\s$). Thus (1) is equivalent to (2). To prove equivalence with (3) and implication of (4), proceed by induction on $0 \leq k \leq K$. The case $k=0$ was treated as fact \ref{L0K} (where $h=g_1=g_{\s^{-1}(1)}=g_{\t_K}$), now assume the lemma is known for $k-1$.

For $b_{K-k}=(b_{K-k+1},a_2,\dots,a_{d_{K-k}})\r$, one has:
$$gb_{K-k}=(g_1a_{\s(1)},\dots,g_{\s^{-1}(1)}b_{K-k+1},\dots,g_{d_{K-k}}a_{\s(d_{K-k})})\s\r. $$

Assume (2) for $g$, then $g_{\s^{-1}(1)}b_{K-k+1} \in L_{k-1}^K(\O)$ for all $b_{K-k+1} \in B_{K-k+1}$, which means (2) for $k-1$ applied to $g_{\s^{-1}(1)}$. By induction hypothesis, $g_{\s^{-1}(1)}$ satisfies (3), which means that it belongs to $\i L_{k-1}^K(\O)$, so $g \in \i L_k^K(\O)$, and $g_{\s^{-1}(1)\t_{K-k+1}\dots \t_K}=g_{\t_{K-k}\t_{K-k+1}\dots \t_K} \in Int(\O)$, proving (3) for $g$. 

Moreover, (2) applied to $g_{\s^{-1}(1)}$ implies, by induction, (4) that $g_{\s^{-1}(1)}b_{K-k+1} \in \i L_{k-1}^K(\O)$ for all $b_{K-k+1} \in B_{K-k+1}$. As $(\s\r)^{-1}(1)=\s^{-1}(\r^{-1}(1))=\s^{-1}(1)$, this shows $gb_{K-k} \in \i L_{K-k}^K(\O)$, which is (4) for $g$.

Conversely, assume (3) for $g$, then $g_{\s^{-1}(1)} \in \i L_{k-1}^K(\O)$, and $g_{\t_{K-k}\t_{K-k+1}\dots \t_K}=g_{\s^{-1}(1)\t_{K-k+1}\dots \t_K} \in Int(\O)$, i.e. (3) for $g_{\s^{-1}(1)}$. As (3) implies (4) for $k-1$, one has $g_{\s^{-1}(1)}b_{K-k+1} \in \i L_{k-1}^K(\O)$ for all $b_{K-k+1} \in B_{K-k+1}$, so $gb_{K-k} \in L_k^K(\O)$ for all $b_{K-k} \in B_{K-k}$, which means (2) for $g$.
\end{proof}

\begin{remark}\label{combin}
The combinatorial description of Remark \ref{remcomb} still applies to an element $g \in \G_{K-k}$ of the form:
$$g=(g_{t_{K-k}\dots t_{K}})(\s_{t_{K-k}\dots t_{K-1}})\dots(\s_{t_{K-k}})\s,$$ 
with $t_{K-k+l} \in \{1,\dots,d_{K-k+l}\}$, $\s_{t_{K-k}\dots t_{K-k+l}}\in A_{K-k+l+1}$ and $g_{t_{K-k}\dots t_{K}} \in \G_{K+1}$. Such an element $g$ belongs to $L_k^K(\O)$ if and only if it satisfies the three following conditions: 
\begin{enumerate}
\item $\forall t_{K-k}\dots t_{K-1}$, the element $g_{t_{K-k}\dots t_{K-1}1}$ is in $\O\subset H_{K+1}$ and the elements $g_{t_{K-k}\dots t_{K-1}2},\dots,g_{t_{K-k}\dots t_{K-1}d_K}$ are in $A_{K+1}$,
\item $\forall t_{K-k}\dots t_{K-2}$, the set: 
\begin{eqnarray*} I(t_{K-k}\dots t_{K-2}) &=& \{T_{K-1}\in\{1,\dots,d_{K-1}\}| \s_{t_{K-k}\dots t_{K-2}T_{K-1}}^{-1}(1)=1\} \\ &=& 
\{T_{K-1}\in\{1,\dots,d_{K-1}\}|g_{t_{K-k}\dots t_{K-2}T_{K-1}} \in \i L_0^K(\O)\subset \G_K\} \end{eqnarray*}
is non-empty.
\item $\forall 2 \leq l \leq k, \forall t_{K-k}\dots t_{K-l}$, the following subset of $\{1,\dots,d_{K-l+1}\}$: 
\begin{eqnarray*} I(t_{K-k}\dots t_{K-l}) &=& \{T_{K-l+1}| \s_{t_{K-k}\dots t_{K-l}T_{K-l+1}}^{-1}(1) \in I(t_{K-k}\dots t_{K-l}T_{k-l+1}) \},  \\ &=& \{ T_{K-l+1}| g_{t_{K-k}\dots t_{K-l}T_{K-l+1}} \in \i L_{l-2}^K(\O)\subset \G_{K-l+2}\}, \end{eqnarray*} 
defined by induction on $l$, is non-empty.
\newline 
\newline
The element $g$ belongs to $\i L_k^K(\O)$ if and only if it satisfies (1), (2), (3) and moreover:
\item $\s^{-1}(1)$ belongs to the set:
$$ I(\emptyset)=\{T_{K-k}|\s_{T_{K-k}}^{-1}(1) \in I(T_{K-k})\} =\{T_{K-k}|g_{T_{K-k}}\in \i L_{k-1}^K(\O)\subset \G_{K-k+1}\}.$$
\newline
The element $g$ belongs to $Int( L_k^K(\O))$ if and only if it satisfies (1), (2), (3), (4) and moreover:
\item $g_{\t_{K-k}\dots \t_K} \in Int(\O)$.
\end{enumerate}
\end{remark}

This description and especially point (5) prove the:

\begin{fact}\label{int}
With respect to the generating set $A_{K-k}\cup B_{K-k}$ of the group $\G_{K-k}$, and the generating set $B_{K+1}$ of the group $H_{K+1}$, one has:
$$|Int(L_k^K(\O))|=|\i L_k^K(\O)|\frac{|Int(\O)|}{|\O|}. $$
\end{fact}

In particular, the set $\i L_k^K(\O)$ is precisely the interior $Int(L_k^K(\O))$ when $Int(\O)=\O$. This happens when $H_{K+1}$ (hence $H_0$) is finite.

For $0\leq k \leq K$, set $\frac{|\i L_k^K(\O)|}{|L_k^K(\O)|}=1-\e_k$. The number $\e_k$ will be denoted $\e^K_k$ later on to emphasize the dependance on $K$. Lemma \ref{delta} generalizes as:

\begin{lemma}\label{receps}
The sequence $(\e_k)_{0\leq k \leq K}$ satisfies $\e_0=1-\frac{1}{d_K}$ and:
$$1-\e_{k+1}=\frac{1-\e_k}{1-\e_k^{d_{K-k-1}}}. $$
\end{lemma}

\begin{proof}
Given a subset $I \subset \{1,\dots,d_{K-k-1}\}$, denote:
$$J_I=\{g=(g_1,\dots,g_{d_{K-k-1}})\s|\forall T \in I,g_T \in \i L_k^K(\O) \textrm{ and }\forall t \notin I, g_t \in L_k^K(\O)\setminus \i L_k^K(\O) \}. $$
By definition, $L_{k+1}^K(\O)$ is the disjoint union $L_{k+1}^K(\O)=\sqcup_{|I| \geq 1}J_I$.

As in the proof of Lemma \ref{delta}, one has for $i=|I|$:
\begin{eqnarray*}
|J_I| &=& |A_{K-k-1}||L_k^K(\O)|^{d_{K-k-1}}(1-\e_k)^i\e_k^{d_{K-k-1}-i}, \\
|J_I\cap \i L_{k+1}^K(\O)| &=&\frac{i}{d_{K-k-1}}|J_I|.
\end{eqnarray*}
Again by use of the mean of binomial distribution, get:
\begin{eqnarray*}
|\i L_{k+1}^K(\O)|&=& \sum_{i=1}^{d_{K-k-1}} C_{d_{K-k-1}}^i (1-\e_k)^i\e_k^{d_{K-k-1}-i}\frac{i}{d_{K-k-1}}|L_k^K(\O)|^{d_{K-k-1}}|A_{K-k-1}|
\\ &=& (1-\e_k)|L_k^K(\O)|^{d_{K-k-1}}|A_{K-k-1}|, \\
|L_{k+1}^K(\O)| &=& \sum_{i=1}^{d_{K-k-1}} C_{d_{K-k-1}}^i (1-\e_k)^i\e_k^{d_{K-k-1}-i}|L_k^K(\O)|^{d_{K-k-1}}|A_{K-k-1}|\\
&=&(1-\e_k^{d_{K-k-1}})|L_k^K(\O)|^{d_{K-k-1}}|A_{K-k-1}|.
\end{eqnarray*}
This proves the lemma.
\end{proof}

\begin{lemma}\label{epsilon}
If $\frac{d_k}{\log k}\longrightarrow 0$, then $\e_K^K \longrightarrow 0$.

If $d_k \leq D$ for all $k$, then $\e_K^K=O(K^{-\eta})$ for all $\eta <\frac{1}{D-1}$.
\end{lemma}

First check the elementary:

\begin{fact}\label{f}
Let $f(D,\e)=\frac{1-\e^{D-1}}{1-\e^D}$, for $D \geq 2$ and $\e \in (0,1)$. Then for fixed $D$, the function $f(D,\e)$ is decreasing with $\e$, and for fixed $\e$, the function $f(D,\e)$ is increasing with $D$.
\end{fact}

\begin{proof} Compute derivatives:
$$(1-\e^D)^2\frac{\partial f}{\partial \e}(D,\e)=\e^{D-2}(1-\e)(\e^{D-1}+\dots+\e^2+\e-(D-1))<0, $$
$$(1-\e^D)^2\frac{\partial f}{\partial D}(D,\e)=\e^{D-1}(\e-1)\log \e>0. $$
\end{proof}

\begin{proof}[Proof of Lemma \ref{epsilon}]
For a fixed $K$, and $0 \leq k \leq K$, set $D_k=d_{K-k}$, and $D(K)=\max_{0\leq k \leq K}\{d_k\}=o(\log K)$. By Lemma \ref{receps}, one has:
$$\e_{k+1}=\e_k\frac{1-\e_k^{D_{k+1}-1}}{1-\e_k^{D_{k+1}}}=\e_k f(D_{k+1},\e_k).$$
By fact \ref{f}, as long as $\e_k \geq E$, one has:
$$\e_{k+1}\leq \e_k f(D_{k+1},E)\leq \e_k f(D(K),E), $$
so $\e_K=\e_K^K\leq \max \{E,f(D(K),E)^K\}$ for any $E \in (0,1)$.
Now consider a sequence $E_K \longrightarrow 0$ so that $|D(K)\log E_K|=o(\log K)$ (it exists). One has:
\begin{eqnarray*}
f(D(K),E_K)^K &=& \exp K( \log(1-E_K^{D(K)-1})-\log(1-E_K^{D(K)}) ),\\
&=&\exp (-KE_K^{D(K)-1}+O(KE_K^{D(K)})) \longrightarrow 0,
\end{eqnarray*}
because $KE_K^{D(K)-1} \longrightarrow +\infty$. This shows $\e_K^K\longrightarrow 0$.

If moreover $d_k \leq D$, take $E_K=K^{-\eta}$ with $\eta < \frac{1}{D-1}$, then:
$$f(D,E_K)^K=\exp(-K^{1-\eta(D-1)}+O(K^{1-\eta D}))=o(K^{-\eta}), $$
so $\e_K^K=O(K^{-\eta})$.
\end{proof}

\begin{proof}[Proof of Theorem \ref{MT}]
By Fact \ref{int}, one has:
$$\frac{|Int(L_K^K(\O))|}{|L_K^K(\O)|}=\frac{|\i L_K^K(\O)|}{|L_K^K(\O)|}\frac{|Int(\O)|}{|\O|}=(1-\e_K^K)\frac{|Int(\O)|}{|\O|}. $$
As the group $H_{K+1}$ is amenable by Fact \ref{fact2}, the set $\O$ can be chosen with $\frac{|Int(\O)|}{|\O|}$ arbitrarily close to $1$. By Lemma \ref{epsilon}, this shows that there exists a sequence of sets $\O_K \subset H_{K+1}$ so that the sets $L_K^K(\O_K)\subset \G_0$ form a Folner sequence. 
\end{proof}

\section{Examples of groups with property $\Dc\Pc$}\label{examples}

\subsection{Alternate directed groups} 

Given a sequence $\bar d=(d_i)_{i\in\N}$ of integers $d_i \geq 2$, set:
$$AT_i=AT(d_i,d_{i+1})=(\Ac_{d_{i+1}}\times \dots \times \Ac_{d_{i+1}})\rtimes \Ac_{d_i-1}=\Ac_{d_{i+1}}\wr \Ac_{d_i-1}, $$
where $\Ac_d$ is the alternate group of even permutations of the set $\{1,\dots,d\}$, there are $d_i-1$ factors in the product (indexed by $\{2,\dots,d_i\}$), and $\Ac_{d_i-1}$ acts by permuting these factors. Consider the countable infinite direct product:
$$H_{\bar d}^{alt}=\prod_{i=0}^\infty AT_i=\prod_{i=0}^\infty \Ac_{d_{i+1}}\wr \Ac_{d_i-1}. $$
Its elements are denoted as sequences $h=(h_i)_{i=0}^\infty$ with $h_i=(a_{i,2},\dots,a_{i,d_i})\r_i \in AT_i$.

The group $H_{\bar d}^{alt}$ acts faithfully on the spherically homogeneous rooted tree $T_{\bar d}$ in the direction of the ray $1^\infty$, where under the canonical isomorphism (\ref{iso}), one has:
$$(h_i)_{i=0}^\infty=((h_i)_{i=1}^\infty,a_{0,2},\dots,a_{0,d_0})\r_0, $$
where $\r_0 \in \Ac_{d_0-1}\simeq Fix_{\Ac_{d_0}}(1)$. Inductively under isomorphism $Aut(T_{s^{k}\bar d}) \simeq Aut(T_{s^{k+1} \bar d})\wr S_{d_k}$, one has $(h_i)_{i=k}^\infty=((h_i)_{i=k+1}^\infty,a_{k,2},\dots,a_{k,d_k})\r_k$.

On the other hand, the group $\Ac_{d_0}$ acts on $T_{\bar d}$ by rooted automorphisms:
$$\Ac_{d_0}\ni a=(e,\dots,e)a. $$

\begin{definition} An alternate directed group $G$ is a subgroup of $Aut^{alt}(T_{\bar d})$ with generating set $A \cup H$, with $A \subset \Ac_{d_0}$ and $H \subset H_{\bar d}^{alt}$. Denote:
$$G(A,H)=\langle A \cup H\rangle <Aut^{alt}(T_{\bar d}). $$
\end{definition}

When the sequence $\bar d$ is constant $d_i=d$, if $A\simeq\Ac_d$ and $H\simeq \Ac_d \wr \Ac_{d-1}$ is diagonaly embedded into the direct product $H_{\bar d}^{alt}$, then $G(A,H)=G_d$ is the alternate mother group of section \ref{first}. Directed groups (not necessarily alternate) satisfy the same definition without requirement that the permutations involved are even, that is with $S_d$ instead of $\Ac_d$ and $H_{\bar d}=\prod_{i=0}^\infty S_{d_{i+1}}\wr S_{d_i-1}$ instead of $H_{\bar d}^{alt}$ (see \cite{Bri1}, \cite{Bri3}).

\subsection{Case of bounded valency} In this section, assume that the sequence $\bar d$ is bounded $5 \leq d_i \leq D$. Let $B \subset H_{\bar d}^{alt}$ be a finite subset, and denote its elements by $\b=(\b_i)_{i=0}^\infty\in H_{\bar d}^{alt}$. Then for each $i$, the set $\{\b_i,\b \in B\}$ is a $B$-indexed subset of $AT_i=AT(d_i,d_{i+1})$. As the valency sequence $\bar d$ is bounded, there is a finite set of pairs:
$$\{(AT(s),\{\b(s),\b \in B\}),s \in J\}, $$
such that for any $i$, there exists $s(i)$ in the finite set $J$ with $(AT_i,\{\b_i,\b \in B\})=(AT(s(i)),\{\b(s(i)),\b \in B\})$, as pairs of finite groups with $B$-indexed subsets.

This provides an isomorphism:
$$H_{\bar d}^{alt}>H=\langle \b,\b\in B\rangle\simeq \langle (\b(s))_{s\in J},\b \in B \rangle< \prod_{s \in J} AT(s). $$
The group $H$ is said saturated if $H=\prod_{s \in J}AT(s)$. (Mind a difference with the notion of saturation in \cite{Bri1} and \cite{Bri3}, where it was only required that $H$ surjects on each factor $AT(s)$. The present condition is slightly stronger.) Finiteness of $J$ shows the:
\begin{fact}\label{saturation}
If $\bar d$ is bounded, any finitely generated subgroup of $H_{\bar d}^{alt}$ is contained in a finite saturated subgroup $H$. 
\end{fact}

The following proposition will permit to show amenability of all directed groups acting on a tree of bounded valency.

\begin{proposition}\label{dpbounded}
Let $\bar d$ be a bounded sequence of integers $d_i \geq 5$. If $H<H_{\bar d}^{alt}$ is a finite saturated subgroup, then the alternate directed group $G(\Ac_{d_0},H)<Aut(T_{\bar d})$ has property $\Dc\Pc$.
\end{proposition}

\begin{proof}
Set $A_k=\Ac_{d_k}$, $H_k=\{(h_i)_{i=k}^\infty| (h_i)_{i=0}^\infty \in H\}$, define $\G_k=G(A_k,H_k)<Aut^{alt}(T_{s^k \bar d})$ and check that the sequence $\{\G_k\}_{k \in \N}$ has property $\Dc\Pc$. In order to ease notations, we treat the case $k=0$, the general case is similar.


The only non-trivial point in order to verify the conditions of Definition \ref{defdp} is surjectivity of the isomorphism:
$$\f_0:G(\Ac_{d_0},H) \longrightarrow G(\Ac_{d_1},H_1)\wr \Ac_{d_0}. $$

Given $h=(h_i)_{i=0}^\infty$ in $H_{\bar d}$ with $h_i=(a_{i,2},\dots,a_{i,d_i})\r_i$, set:
$$h(2)=((a_{i,2},e,\dots,e)e)_{i=0}^\infty \textrm{, and } h(\emptyset)=((e,\dots,e)\r_i)_{i=0}^\infty. $$
In each factor $AT(s)=\Ac_{d'(s)}\wr \Ac_{d(s)-1}$, the subset
$$\{ (a_2,e,\dots,e)|a_2 \in \Ac_{d'(s)}\} \cup \{(e,\dots,e)\r|\r \in \Ac_{d(s)-1}\} $$
generates the group $AT(s)$. Thus by saturation
$$\langle h(2),h \in H \rangle \simeq \prod_{s \in J} \Ac_{d'(s)} \times \{e\} \times \dots \times \{e\}  \textrm{, and } \langle h(\emptyset),h \in H \rangle \simeq \prod_{s \in J} \Ac_{d'(s)}.$$
So saturation shows that the subsets $H(2)=\{h(2),h \in H\}$ and $H(\emptyset)=\{h(\emptyset),h \in H\}$ are subgroups of $H$, and moreover $\langle H(2) \cup H(\emptyset)\rangle=H$.

The proofs of Fact \ref{perfect1} and Proposition \ref{wr} apply directly, replacing the generators $b_2=b(\a_2,e,\dots,e,e_A)$ and $b_\emptyset=b(e,\dots,e,\r)$ by $h(2)$ and $h(\emptyset)$ respectively.
\end{proof}

Let $\s$ be a permutation of the set $\{1,\dots,d\}$. Denote $\s'$ another copy of $\s$ acting on the set $\{d+1,\dots,2d\}$ by $\s'(t)=\s(t-d) +d$, and consider the embedding $a:S_d \hookrightarrow \Ac_{2d}$ given by $a(\s)=\s\s'$. It can be extended to furnish:
$$a:Aut(T_{\bar d}) \rightarrow Aut^{alt}(T_{\bar{2d}}), $$
an embedding of the group of automorphisms of the tree $T_{\bar d}$ into the group of alternate automorphisms of the tree $T_{\bar{2d}}$.

Indeed, let $\g\in Aut(T_{\bar d})$ be described by a family of permutations $\{\s_v\}_{v \in T_{\bar d}}$, where $\s_v \in S_{d_k}$ for every $v=t_1\dots t_k$ in $T_{\bar d}$. The automorphism $a(\g)$ is described by a family of permutations $\{a(\g)_v\}_{v \in T_{\bar{2d}}}$ given by $a(\g)_v=a(\g_v) \in \Ac_{2d_k}$ for $v=t_1\dots t_k$ in $T_{\bar d} \subset T_{\bar{2d}}$ and $a(\g)_v=e$ for $v \in T_{\bar{2d}} \setminus T_{\bar{d}}$.

\begin{fact}\label{directalt}
Directed elements have directed image under $a$, i.e. $a(H_{\bar d}) \subset H_{\bar{2d}}^{alt}$. In particular, the mother group of degree 0 acting on a $d$-regular tree embeds in the alternate mother group $G_{2d}$ acting on a $2d$-regular tree.
\end{fact}

\begin{proof} 
As a shortcut denote $1^k$ for the sequence $11\dots 1$ with $k$ ones. By definition, an automorphism $\g$ is directed if and only if $\s_{1^k}\in Fix_{S_{d_k}}(1)\simeq S_{d_k-1}$ and $\s_v=e$ if $v$ is not of the form $1^{k-1}t$ for some $t$ in $\{1,\dots,d_k\}$. This is still the case for $a(\g)$.
\end{proof}

The following result from \cite{Bri1} can now be reproved.
\begin{corollary}
Directed groups acting on a tree of bounded valency are amenable.
\end{corollary}

\begin{proof}
Let $\G$ be a directed group, with generating set $S \cup H$ where $S\subset S_{d_0}$ and $H\subset H_{\bar d}$. By fact \ref{directalt}, the group $a(\G) < Aut^{alt}(T_{\bar{2d}})$ is alternate and directed. By fact \ref{saturation}, it can be included in a directed, alternate and saturated subgroup of $Aut^{alt}(T_{\bar{2d}})$, which has property $\Dc\Pc$ by Proposition \ref{dpbounded}, hence $a(\G)$ is amenable by Theorem \ref{MT}, since $\bar{2d}$ is bounded and $H_0$ finite. The group $\G$ is also amenable as a subgroup.
\end{proof}

\begin{corollary}[Main theorem in \cite{BKN}]
Automata groups with bounded activity are amenable.
\end{corollary}

\begin{proof}
By Theorem 3.3 in \cite{BKN}, an automata group $\G$ with bounded activity is a subgroup of the alternate mother group of degree 0 acting on a $d$-regular tree for $d$ large enough. By Fact \ref{directalt}, $\G$ is a subgroup of $G_{2d}$, hence is amenable by Theorem \ref{deg0}.
\end{proof}

\subsection{Examples with unbounded valency} This section aims at constructing examples of groups with property $\Dc\Pc$ for which the sequence $\bar d$ of fact \ref{tree} is unbounded.

Let $H$ be a finitely generated, residually finite, perfect group with a sequence of normal subgroups $(N_i)_{i\geq 0}$ of finite index so that each quotient $A_i=H/N_i$ is perfect, acting faithfully and transitively on a finite set $\{1,\dots,d_i\}$ of size $d_i \geq 2$. For $h$ in $H$, denote $a_i(h)=hN_i \in A_i$. Assume moreover that there exists $\t_i \in A_i$ such that $\t_i(1)=1$ and $\t_i^{-1}(2) \notin \{1,2\}$.

To the group $H=H_0$ together with subgroup sequence $(N_k)_{k \geq 0}$ is associated an action on the rooted tree $T_{\bar d}$ of valency sequence $\bar d=(d_k)_{k\geq 0}$, denoted $b_0:H_0\rightarrow Aut(T_{\bar d})$, given by the portrait $(b_0(h))_{1^{k-1}2}=a_k(h)$ and $(b_0(h))_v=e$ if $v$ is not of the form $1^{k-1}2$ for $k\geq 1$ (notation $1^k$ is a shortcut for $11\dots 1$ with $k$ ones). 

More generally, to $H_i=H$ together with the sequence $(N_k)_{k \geq i}$ is associated an action on $T_{s^i \bar d}$ denoted $b_i : H_i \rightarrow Aut(T_{s^i \bar d})$, given by the portrait $(b_i(h))_{1^{k-1}2}=a_{k+i}(h)$ and $(b_i(h))_v=e$ if $v$ is not of the form $1^{k-1}2$.

The group $A_i=H_i/N_i=H/N_i$ also acts on $T_{s^i \bar d}$ as a rooted automorphism acting on $\{1,\dots ,d_i\}$, i.e. $a_i(h)=(e,\dots,e)a_i(h)$. Set $\G_i=\langle A_i \cup H_i\rangle < Aut(T_{s^i \bar d})$, with the $b_i$ action.

\begin{fact}\label{classinf}
The sequence of groups $\{\G_i\}_{i \in \N}$ has property $\Dc\Pc$.
\end{fact}

\begin{proof}
In the wreath product isomorphism $\f_i$ of (\ref{iso}) for $T_{s^i \bar d}$, one has:
\begin{eqnarray}\label{surj}\f_i(b_{i}(h))=(b_{i+1}(h),a_{i+1}(h),e,\dots,e).\end{eqnarray}
Thus it induces an embedding $\f_i: \G_i \hookrightarrow \G_{i+1} \wr A_i$. The only non-trivial point in the conditions of Definition \ref{defdp} is to check that this embedding is onto. 

As $\f_i(b_i(h')^{\t_i})=(b_{i+1}(h'),e,\dots,a_{i+1}(h'),\dots,e)$, one has $\f_i([b_i(h')^{\t_i},b_i(h)])=(b_{i+1}([h',h]),e,\dots,e)$. As $H$ is perfect, the image contains $H_{i+1} \times \{e\} \times \dots \times \{e\}$, and also $\{e\} \times A_{i+1} \times \{e\} \times \dots \times \{e\}$ by (\ref{surj}). Moreover, the image contains $A_i$ rooted which has a transitive action on $\{1,\dots,d_i\}$. Thus $\f_i(\G_i)$ finally contains $\G_{i+1}\wr A_i$.
\end{proof}


As an example of such a finitely generated, residually finite, perfect group $H$, one may take the alternate mother group $G_d$ of section \ref{first} for $d\geq 6$ (for which both finite generating subgroups $A$ and $B$ are perfect). This group satisfies $G_d \simeq G_d \wr \Ac_d $. Its finite index normal subgroups are:
$$St_j=\ker (G_d \rightarrow \Ac_d \wr \dots \wr \Ac_d), $$
where the $j$ factors in the iterated wreath product are obtained by iteration of the above isomorphism. The group $St_j$ is called stabilizer of level $j$ of the group $G_d$. The quotient $G_d/St_j$ is acting transitively on level $j$, which is the set $\{1,\dots,d\}^j$. By \cite{Neu}, these stabilizers $St_j$ are the only finite index normal subgroups of $G_d$.

For an arbitrary function $j:\N\rightarrow \N$, take $N_k=St_{j(k)}$ as a sequence of normal subgroups. The group $\G_0$ defined by $H=G_d$ together with the function $j(k)$ has property $\Dc\Pc$ by Fact \ref{classinf}. It is amenable when $d^{j(k)}$ is sublogarithmic by Theorem \ref{MT}. Note that in the construction above, one could use any group of Proposition \ref{dpbounded} with $d_i \geq 6$ instead of $G_d$.

\textbf{Acknowledgements.} I wish to thank Prof. Tsuyoshi Kato and Prof. Andrzej Zuk for interesting discussions and comments. I also wish to thank the anonymous referee for valuable suggestions. This work was realized during a JSPS Postdoctoral Fellowship Program (PE11006) at Kyoto University.

\end{document}